\documentclass[12pt,reqno]{amsart}

\usepackage{amsfonts,amsthm,amsmath}
\usepackage{amsaddr}
\newcommand{\numberset}{\mathbb} 
\newcommand{\R}{\numberset{R}}

\newcommand{\Uk}{\bigcup_{k=1}^n B(a_k,\varepsilon)}
\newcommand{\Ude}{\bigcup_{k=1}^n B(a_k,\varepsilon+\delta)}

\usepackage[a4paper, total={125mm, 185mm}]{geometry}

\usepackage{verbatim}
\usepackage{setspace}

\makeatletter

\numberwithin{equation}{section}
\newtheorem{thm}{\indent\bf {Theorem}}[section]

\theoremstyle{definition}

\begin{document}

\def\author@andify{%
   \nxandlist {\unskip ,\penalty-1 \space\ignorespaces}%
     {\unskip {} }%
     {\unskip ,\penalty-2 \space }%
}
\title[Multipolar Hardy inequalities]{Multipolar Hardy inequalities and mutual interaction of the poles}

\author[A. Canale]{Anna Canale}
\address{Dipartimento di Matematica, 
Università degli Studi di Salerno, Via Giovanni Paolo II, 132, 84084 Fisciano
(Sa), Italy.
}

\thanks{
The author is member of the Gruppo Nazionale per l'Analisi Matematica, 
la Probabilit\'a e le loro Applicazioni 
(GNAMPA) of the Istituto Nazionale di Alta Matematica (INdAM)}



\begin{abstract}

In this paper we state the weighted Hardy inequality
\begin{equation*}
c\int_{{\mathbb R}^N}\sum_{i=1}^n \frac{\varphi^2 }{|x-a_i|^2}\, \mu(x)dx\le  
\int_{{\mathbb R}^N} |\nabla\varphi|^2 \, \mu(x)dx 
+k \int_{\R^N}\varphi^2 \, \mu(x)dx
\end{equation*}
for any $ \varphi$ in a weighted Sobolev spaces, with $c\in]0,c_o[$ where $c_o=c_o(N,\mu)$ is the optimal constant, $a_1,\dots,a_n\in \R^N$, $k$ is a constant depending on $\mu$. 

We show the relation between $c$ and the closeness to the single pole. To this aim we analyze in detail the difficulties to be overcome to get the inequality.

\end{abstract}

\maketitle

\bigskip

\section{Introduction}
\label{sec:1}

The paper is devoted to multipolar Hardy inequalities with weight in $\R^N$, $N\ge 3$, with a class of weight functions wide enough. The main difficulties to get the inequalities in the multipolar case rely on the mutual interaction among the poles.

The interest in weighted Hardy inequalities is due to the applications to the study of Kolmogorov operators
\begin{equation}\label{L Kolmogorov}
Lu=\Delta u+\frac{\nabla \mu}{\mu}\cdot\nabla u,
\end{equation}
defined on smooth functions, $\mu>0$ is a probability density on $\R^N$, perturbed by inverse square potentials of multipolar type and of the related evolution problems
$$
(P)\quad\left\{\begin{array}{ll}
\partial_tu(x,t)=Lu(x,t)+V(x)u(x,t),\quad \,x\in {\mathbb R}^N, t>0,\\
u(\cdot ,0)=u_0\geq 0\in L^2(\R^N, \mu(x)dx).
\end{array}
\right. $$
In the case of a single pole and of the Lebesgue measure there is a very huge literature on this topic. For the classical Hardy inequality we refer, for example,  to \cite{Hardy20,Hardy25,HLP,D,KMP,KO}. 

We focus our attention on multipolar Hardy's inequalities.

When $L$ is the Schr\"odinger operator with multipolar inverse square potentials we can find some reference result in literature.

In particular, for the operator
$$
\mathcal{L}=-\Delta-\sum_{i=1}^n\frac{c_i}{|x-a_i|^2},
$$
$n\ge2$, $c_i\in \R$, for any $i\in \{1,\dots, n\}$, V. Felli, E. M. Marchini and S. Terracini in
\cite{FelliMarchiniTerracini} proved that the associated quadratic form
$$
Q(\varphi):=\int_{\R^N}|\nabla \varphi |^2\,dx
-\sum_{i=1}^n c_i\int_{{\mathbb R}^N}\frac{\varphi^2}{|x-a_i|^2}\,dx
$$
is positive if $\sum_{i=1}^nc_i^+<\frac{(N-2)^2}{4}$, $c_i^+=\max\{c_i,0\}$, 
conversely if
$\sum_{i=1}^nc_i^+>\frac{(N-2)^2}{4}$ there exists a configuration of poles such that $Q$ is not positive.
Later R. Bosi, J. Dolbeaut and M. J. Esteban in \cite{BDE} proved that 
for any $c\in\left(0,\frac{(N-2)^2}{4}\right]$ there exists 
a positive constant $K$ such that the multipolar Hardy inequality
\begin{equation*}
	c\int_{{\R}^N}\sum_{i=1}^n\frac{\varphi^2 }{|x-a_i|^2}\, dx\le 
	\int_{{\R}^N} |\nabla\varphi|^2 \, dx\\
	+ K \int_{\R^N}\varphi^2 \, dx 
\end{equation*}
holds for any $\varphi \in H^1(\R^N)$.
C. Cazacu and E. Zuazua in \cite{CazacuZuazua}, improving a result stated in 
\cite{BDE}, obtained the inequality 

$$
\frac{(N-2)^2}{n^2}\sum_{\substack{i,j=1\\ i< j}}^{n}
	\int_{\R^N}\frac{|a_i-a_j|^2}{|x-a_i|^2|x-a_j|^2}\varphi^2\,dx
	\le\int_{\R^N}|\nabla \varphi|^2\,dx,
$$
for any $\varphi\in H^1(\R^N)$ with $\frac{(N-2)^2}{n^2}$ optimal constant (see also \cite{Cazacu} for estimates in bounded domains).
 
For Ornstein-Uhlenbeck type operators 
$$
Lu=\Delta u - \sum_{i=1}^{n}A(x-a_i)\cdot \nabla u,
$$
perturbed by multipolar
inverse square potentials 
\begin{equation*}
V(x)=\sum_{i=1}^n \frac{c}{|x-a_i|^2},\quad c>0, \quad a_1\dots,a_n\in \R^N,
\end{equation*}
weighted multipolar Hardy inequalities with optimal constant
and related existence and nonexistence of solutions to the problem (P)
were stated in \cite{CP} following Cabr\'e-Martel's approach in \cite{CabreMartel}, 
with $A$ a positive definite real Hermitian $N\times N$ matrix, $a_i\in \R^N$, $i\in \{1,\dots , n\}$. 
In such a case, the invariant measure for these operators is the Gaussian measure
$\mu_A (x) dx =\kappa e^{-\frac{1}{2}\sum_{i=1}^{n}\left\langle A(x-a_i), x-a_i\right\rangle }dx$,
with a normalization constant $\kappa$. The technique used to get the inequality applies to the Gaussian functions and it allows to get the result in a simple way. More delicate issue is to prove the optimality of the constant.

In \cite{CPT2} these results have been extended to Kolmogorov operators with a more general drift term which force us to use different methods. 

The result stated in \cite{CazacuZuazua} has been extended to the weighted multipolar case in \cite {CA multipolar Gold}.


In this paper we improve a result in \cite{CPT2}. In particular we state that it holds

\begin{equation}\label{ineq intro}
c\int_{{\mathbb R}^N}\sum_{i=1}^n \frac{\varphi^2 }{|x-a_i|^2}\, \mu(x)dx\le  
\int_{{\mathbb R}^N} |\nabla\varphi|^2 \, \mu(x)dx 
+k \int_{\R^N}\varphi^2 \, \mu(x)dx
\end{equation}
for any $ \varphi \in H^1_\mu$, with $c\in]0,c_o[$ where $c_o=c_o(N,\mu)$ is the optimal constant, showing the relation between $c$ and the closeness to the single pole and improving the constant $k$ in the estimate. The proof initially uses {\sl the vector field method} (see \cite{M}) extended to the weighted case. Then we overcome the difficulties related to the mutual interaction between the poles  
emphasizing this relation. 

The class of weight functions satisfy conditions of quite general type, in particular integrability conditions to get a density result which allows us to state inequality (\ref{ineq intro}) for any function in the weighted Sobolev space. Weights of this type were considered in 
\cite{CPT1,CAHardytype,CAimproved,CAimproved2} in the case of a single pole.

Until now,  we can achieve the optimal constant on the left-hand side in (\ref{ineq intro}) using the IMS truncation method \cite{Morgan,SimonIMS} (see \cite{BDE} in the case of Lebesgue measure and \cite{CPT2} in the weighted case). As a counterpart, the estimate is not very good when the constant $c$ is close to the constant $\frac{c_o(N,\mu)}{n}$ as observed in \cite{BDE} in the unweighted case.

The paper is organized as follows. In Section \ref{sec:2} we consider the weight functions with an example. In Section \ref{sec:3} we show a preliminar result introducing suitable estimates useful to state the main result in Section \ref{sec:4}.

\bigskip

\section{Weight functions}
\label{sec:2}

Let $\mu\ge 0$ be a weight function on $\R^N$. We define the weighted Sobolev space 
$H^1_\mu=H^1(\R^N, \mu(x)dx)$
as the space of functions in $L^2_\mu:=L^2(\R^N, \mu(x)dx)$ whose weak derivatives belong to
$L_\mu^2$.

In the proof of weighted estimates we make us of vector field method introduced in \cite {M} 
in the case of a single pole and extended to the multipolar case in \cite {CPT2}. 
To this aim we define the vector value function
$$
F(x)=\sum_{i=1}^n \beta\, \frac{x-a_i}{|x-a_i|^2} \mu(x), \qquad \beta>0.
$$
The class of weight functions $\mu$ that we consider fulfills the conditions:

\begin{itemize}
\item[$H_1)$] 
\begin{itemize}
\item[$i)$] $\quad \sqrt{\mu}\in H^1_{loc}(\R^N)$;
\item[$ii)$]  $\quad \mu^{-1}\in L_{loc}^1(\R^N)$;
\end{itemize}
\item [$H_2)$] there exists constants $C_\mu, K_\mu\in \R$, $K_\mu>2-N$, such that 
it holds
\begin{equation*}
-\beta\sum_{i=1}^n\frac{(x-a_i)}{|x-a_i|^2}\cdot\frac{\nabla\mu}{\mu}\le
C_\mu+ K_\mu\sum_{i=1}^n\frac{\beta}{|x-a_i|^2}.
\end{equation*}
\end{itemize}
Under the hypotheses $i)$ and $ii)$ in $H_1)$ the 
space $C_c^{\infty}(\R^N)$ is dense in $H_{\mu}^1$ (see e.g. \cite{T}).  
So we can regard
$H_{\mu}^1$ as the completion of $C_c^{\infty}(\R^N)$ with respect to the Sobolev norm
$$
\|\cdot\|_{H^1_\mu}^2 := \|\cdot\|_{L^2_\mu}^2 + \|\nabla \cdot\|_{L^2_\mu}^2.
$$
The density result allows us to get the weighted inequalities for any function in $H^1_\mu$. 
As a consequence of the assumptions on $\mu$,
we get $F_j$, $\frac{\partial F_j}{\partial x_j}\in L_{loc}^1(\R^N)$, where 
$F_j(x)=\beta \sum_{i=1}^{n}\frac{(x-a_i)_j}{|x-a_i|^2}\mu(x)$.
This allows us to integrate by parts in the proof of the Teorem \ref{wHi} in Section 
\ref{sec:3}. 

An example of weight function satisfying $H_2)$ is 

$$
\mu(x)=\prod_{j=1}^n\mu_j (x)=
e^{-\delta\sum_{j=1}^{n}|x-a_j|^2}, \qquad \delta\ge0.
$$
Let us see it in detail without worrying about the best estimates. We get

$$
\frac{\nabla\mu}{\mu}=\sum_{j=1}^{n}\frac{\nabla\mu_j}{\mu_j}=
-2\delta\sum_{j=1}^{n} (x-a_j).
$$ 
So, taking in mind the left-hand-side in $H_2)$, 

\begin{equation*}
-\beta\sum_{i=1}^{n}\frac{(x-a_i)}{|x-a_i|^2}\cdot\frac{\nabla \mu}{\mu}= 
2\beta\delta\sum_{i,j=1}^{n}
\frac{(x-a_i)\cdot(x-a_j)}{|x-a_i|^2}. 
\end{equation*}

\noindent We estimate the scalar product. In $B(a_k, r_0)$,  for any $k\in\{1, \dots n\}$, we get


\begin{equation}\label{ineq Bk}
\begin{split}
2\beta\delta\sum_{i=1}^{n}&\frac{(x-a_i)}{|x-a_i|^2}\cdot\frac{\nabla \mu}{\mu}=
2\beta\delta\frac{(x-a_k)\cdot(x-a_k)}{|x-a_k|^2}
\\&+
2\beta\delta\sum_{\substack{i\ne k\\j=i}}^{n}\frac{(x-a_i)\cdot(x-a_i)}{|x-a_i|^2}+
2\beta\delta\sum_{j\ne k}\frac{(x-a_k)\cdot(x-a_j)}{|x-a_k|^2}
\\&+
2\beta\delta \sum_{\substack{i\ne k\\j\ne i}}^{n}
\frac{(x-a_i)\cdot(x-a_j)}{|x-a_i|^2}=
J_1+J_2+J_3+J_4.
\end{split}
\end{equation}
So
$$
J_1+J_2=2\beta\delta n.
$$
Since
$$
(x-a_k)\cdot(x-a_j)=\frac{1}{2}\left(
|x-a_k|^2+|x-a_j|^2-|a_k-a_j|^2\right),
$$
$J_3$ and $J_4$ can be estimated as follows.

\begin{equation*}
J_3= \beta\delta \sum_{j\ne k}
\left(1+\frac{|x-a_j|^2-|a_k-a_j|^2}{|x-a_k|^2}\right)
\le\beta\delta \sum_{j\ne k} \left[1+
\frac{(r_0+|a_k-a_j|)^2-|a_k-a_j|^2}{|x-a_k|^2}\right]
\end{equation*}
and
$$
J_4=\beta\delta \sum_{\substack{i\ne k\\j\ne i}}^{n} 
\left(1+\frac{|x-a_j|^2-|a_i-a_j|^2}{|x-a_i|^2}\right)
\le \beta\delta \sum_{\substack{i\ne k\\j\ne i}}^{n} \left[1+
\frac{(r_0+|a_k-a_j|)^2-|a_i-a_j|^2}{(|a_k-a_i|-r_0)^2}\right].
$$
Then for $C_\mu$ large enough and $K_{\mu,r_0}=
\delta\sum_{j\ne k} (r_0^2+2r_0|a_k-a_j|)$
in $B(a_k, r_0)$ the condition $H_2)$ holds. For $x\in\R^N\setminus 
\bigcup_{k=1}^n B(a_k, r_0)$ we obtain

$$
\frac{(x-a_i)\cdot(x-a_j)}{|x-a_i|^2}\le \frac{|x-a_j|}{|x-a_i|}\le \hbox{const}.
$$
In fact, if $|x|>2\max_i |a_i|$,
$$
\frac{|x|}{2}\le |x|-|a_i| \le|x-a_i|\le |x|+|a_i|\le \frac{3}{2}|x|
$$
for any $i$,
so for $|x|$ large enough we get $|x-a_i|\sim |x|$.
Instead if $|x|\le R=2\max_i |a_i|$,
$$
r_0\le|x-a_i|\le |x|+|a_i|\le \frac{3}{2}R
$$
for any $i$. 

For other examples see \cite{CPT2}.

\bigskip

\section{A preliminary estimate}
\label{sec:3}

The next result was stated in \cite{CPT2} (see also \cite{CA multipolar Gold}). We give 
a riformulated version that is functional to our purposes. 
The estimate represents a preliminary weighted Hardy inequality.

\medskip

\begin{thm}\label{wHi with beta}
Let $N\ge 3$ and $n\ge 2$.
Under hypotheses  $H_1)$ and $H_2)$ we get
\begin{equation}\label{ineq beta}
\begin{split}
\int_{\R^N}\sum_{i=1}^{n}&\frac{\beta(N+K_\mu-2)-n\beta^2}{|x-a_i|^2}\varphi^2\,d\mu
\\&+
\frac{\beta^2}{2}\int_{\R^N}\sum_{\substack{i,j=1\\ i\ne j}}^{n}
\frac{|a_i-a_j|^2}{|x-a_i|^2|x-a_j|^2}\varphi^2 \, d\mu
\\&\le 
\int_{\R^N}|\nabla\varphi|^2 \, d\mu+C_\mu\int_{\R^N}\varphi^2 \, d\mu
\end{split}
\end{equation}
for any $ \varphi \in H_\mu^1$.
As a consequence the following inequality holds
\begin{equation}\label{preliminary wHi} 
\begin{split}
c_{N,n,\mu}\int_{\R^N}\sum_{i=1}^{n}&\frac{\varphi^2}{|x-a_i|^2}\,d\mu
\\&+
\frac{c_{N,n,\mu}}{2n}\int_{\R^N}\sum_{\substack{i,j=1\\ i\ne j}}^{n}
\frac{|a_i-a_j|^2}{|x-a_i|^2|x-a_j|^2}\varphi^2 \, d\mu
\\&\le
 \int_{\R^N}|\nabla\varphi|^2 \, d\mu+C_\mu\int_{\R^N}\varphi^2 \, d\mu,
\end{split}
\end{equation}
where $c_{N,n,\mu}=\frac{(N+K_\mu-2)^2}{4n}$ is
the maximum value of the first constant on left-hand side in (\ref{ineq beta}) attained for $\beta=\frac{N+K_\mu-2}{2 n}$.
\end{thm}
The proof of the Theorem \ref{wHi with beta} in \cite {CPT2} is based on the vector field method extended to the multipolar case. 
In \cite{BDE} an estimate similar to (\ref{preliminary wHi}) was obtained in a different way when $\mu =1$.

We observe that inequality (\ref {preliminary wHi}) is an improved inequality with respect 
to the first example of multipolar inequality with weight
\begin{equation*} 
\frac{(N+K_\mu-2)^2}{4n}\int_{\R^N}\sum_{i=1}^{n}\frac{\varphi^2}{|x-a_i|^2}\,d\mu
\le \int_{\R^N}|\nabla\varphi|^2 \, d\mu+C_\mu\int_{\R^N}\varphi^2 \, d\mu,
\end{equation*}
which the natural generalization of the weighted Hardy inequality (see \cite{CPT1})
\begin{equation} \label{ineq unipolar}
\frac{(N+K_\mu-2)^2}{4}\int_{\R^N}\frac{\varphi^2}{|x|^2}\,d\mu
\le\int_{\R^N}|\nabla\varphi|^2 \, d\mu+C_\mu\int_{\R^N}\varphi^2 \, d\mu,
\end{equation}
Now we focus our attention on the second term on the left-hand side in (\ref{ineq beta}).  
For simplicity we put
\begin{equation}\label{def W}
W(x):=\frac{1}{2}\sum_{\substack{i,j=1\\ i\ne j}}^{n}
\frac{|a_i-a_j|^2}{|x-a_i|^2|x-a_j|^2}
\end{equation}


\noindent In $B(a_i, r_0)$, taking into account that
$$
W=\frac{1}{|x-a_i|^2}\sum_{j\ne i}\frac{|a_i-a_j|^2}{|x-a_j|^2}+
\sum_{\substack{k,j\ne i\\ j>k}}^{n}
\frac{|a_k-a_j|^2}{|x-a_k|^2|x-a_j|^2},
$$
we have the following estimates for $W$ from above and from below
\begin{equation*}
\begin{split}
W\le & \frac{1}{|x-a_i|^2}\sum_{j\ne i}\frac{(|a_i-a_j|^2}{(|a_i-a_j|-|x-a_i|)^2}
\\& +
\sum_{\substack{k,j\ne i\\ j>k}}^{n}
\frac{|a_k-a_j|^2}{(|a_i-a_k|-|x-a_i|)^2(|a_i-a_j|-|x-a_i|)^2}
\\& \le
\frac{n-1}{|x-a_i|^2}\frac{d^2}{(d-r_0)^2}+
\sum_{\substack{k,j\ne i\\ j>k}}^{n}
\frac{|a_k-a_j|^2}{(|a_i-a_k|-r_0)^2(|a_i-a_j|-r_0)^2}
\\&\le
\frac{n-1}{|x-a_i|^2}\frac{d^2}{(d-r_0)^2}+c_1
\end{split}
\end{equation*}
and 
\begin{equation}\label{W from below}
\begin{split}
W\ge & \frac{1}{|x-a_i|^2}\sum_{j\ne i}\frac{(|a_i-a_j|^2}{(|a_i-a_j|+|x-a_i|)^2}
\\& +
\sum_{\substack{k,j\ne i\\ j>k}}^{n}
\frac{|a_k-a_j|^2}{(|a_i-a_k|+|x-a_i|)^2(|a_i-a_j|+|x-a_i|)^2}
\\& \ge
\frac{n-1}{|x-a_i|^2}\frac{d^2}{(d+r_0)^2}+
\sum_{\substack{k,j\ne i\\ j>k}}^{n}
\frac{|a_k-a_j|^2}{(|a_i-a_k|+r_0)^2(|a_i-a_j|+r_0)^2}
\\&\ge
\frac{n-1}{|x-a_i|^2}\frac{d^2}{(d+r_0)^2}+c_2.
\end{split}
\end{equation}
\noindent When $x$ tends to $a_i$ we get 
\begin{equation}\label{behav W}
W\sim \frac{n-1}{|x-a_i|^2}
\end{equation}
and, then, taking in mind the inequality (\ref{ineq beta}), we have the asymptotic behaviour
\begin{equation}\label{behav W 2}
\begin{split}
\sum_{i=1}^{n}&\frac{\beta(N+K_\mu-2)-n\beta^2}{|x-a_i|^2}
+
\frac{\beta^2}{2}\sum_{\substack{i,j=1\\ i\ne j}}^{n}\frac{|a_i-a_j|^2}{|x-a_i|^2|x-a_j|^2}
\\&\sim
\left[\beta(N+K_\mu-2)-n\beta^2+\beta^2 (n-1)\right]
\frac{1}{|x-a_i|^2}
\\&=
\left[\beta(N+K_\mu-2)-\beta^2\right]\frac{1}{|x-a_i|^2}.
\end{split}
\end{equation}
The maximum value of the constant on the right-hand side in (\ref{behav W 2}) is 
the best constant in the weighted Hardy inequality with a single pole (see (\ref{ineq unipolar})).

\bigskip

\section{Weighted multipolar Hardy inequality}
\label{sec:4}

The behaviour of the function $W$ in (\ref{behav W}) when $x$ tends to the pole $a_i$ 
leads us to study the relation between the constant on the left-hand side in weighted Hardy 
inequalities and the closeness to the single pole.
The next result emphasizes this relation and improves a similar inequality stated 
in \cite{CPT2} in a different way. 

\medskip

\begin{thm}\label{wHi}
Assume that the conditions  $H_1)$ and $H_2)$ hold.  
Then for any $ \varphi \in H^1_\mu$ we get

\begin{equation}\label{ineq}
c\int_{{\mathbb R}^N}\sum_{i=1}^n \frac{\varphi^2 }{|x-a_i|^2}\, \mu(x)dx\le  
\int_{{\mathbb R}^N} |\nabla\varphi|^2 \, \mu(x)dx 
+k \int_{\R^N}\varphi^2 \, \mu(x)dx
\end{equation}
with $c\in\left]0,c_o(N+K_\mu)\right[$,
where $c_o(N+K_\mu)=\left(\frac{N+K_\mu-2}{2}\right)^2$ optimal constant, and
$k=k(n, d, \mu)$, $d:=\min_{\substack{1\le i,j\le n\\ i\ne j}} |a_i-a_j|/2$.
\end{thm}

\begin{proof}
By density, it is enough to prove (\ref{ineq}) for any
$\varphi \in C_{c}^{\infty}(\R^N)$. 

The optimality of the constant $c_o(N+K_\mu)$ was stated in \cite{CPT2}. 
We will prove the inequality (\ref{ineq}).

We start from the integral
\begin{equation}\label{def div F}
\int_{\R^N}\varphi^2 {\rm div}F \,dx=
\beta\int_{\R^N}\sum_{i=1}^n\left[\frac{N-2}{|x-a_i|^2} \mu(x) +
\frac{(x-a_i)}{|x-a_i|^2}\cdot\nabla\mu\right]\varphi^2 dx
\end{equation} 
and integrate by parts getting, through H{\"o}lder's and Young's inequalities, 
the first following inequality

\begin{equation}\label{div F}
\begin{split}
\int_{\R^N}&\varphi^2 {\rm div}F \,dx=
-2\int_{\R^N}^{}\varphi F\cdot\nabla\varphi \, dx
\\&
\le 2\left(\int_{\R^N}|\nabla \varphi|^2 \, \mu(x)dx\right)^{\frac{1}{2}}
\left[\int_{\R^N}\left| \sum_{i=1}^{n} \frac{ \beta\,(x-a_i)}{|x-a_i|^2} \right|^2 
\,\varphi^2\, \mu(x)dx\right]^{\frac{1}{2}}
\\&
\le \int_{\R^N}|\nabla \varphi|^2 \, \mu(x)dx+
\int_{\R^N}\left| \sum_{i=1}^{n} \frac{\beta\,(x-a_i)}{|x-a_i|^2} \right|^2 \,\varphi^2\, \mu(x)dx.
\end{split}
\end{equation}
So from (\ref{def div F}), using the estimate (\ref{div F}), we get 
\begin{equation}\label{first part}
\begin{split}
\int_{\R^N}\sum_{i=1}^n \frac{\beta(N-2)}{|x-a_i|^2}\varphi^2&\mu(x)dx
\le  \int_{\R^N}|\nabla \varphi|^2 \, \mu(x)dx
\\&+
\int_{\R^N} \sum_{i=1}^{n} \frac{\beta^2}{|x-a_i|^2}  \,\varphi^2\, \mu(x)dx
\\&
+\int_{\R^N} \sum_{\substack{i,j=1\\j\ne i}}^n\frac{ \beta^2\,(x-a_i)
\cdot (x-a_j)}{|x-a_i|^2|x-a_j|^2} \,\varphi^2\, \mu(x)dx
\\&
-\beta
\int_{\R^N}\sum_{i=1}^n\frac{(x-a_i)}{|x-a_i|^2}\cdot\nabla\mu\,\varphi^2 dx.
\end{split}
\end{equation}
Let $\varepsilon>0$ small enough and $\delta>0$ such that 
$\varepsilon+\delta<\frac{d}{2}$. The next step is to estimate the integral of the mixed term that comes out 
the square of the sum in (\ref{first part}) by writing 

\begin{equation}\label{integr mixed term}
\begin{split}
\int_{\R^N}& \sum_{\substack{i,j=1\\j\ne i}}^n
\frac{\beta^2\,(x-a_i)\cdot (x-a_j)}{|x-a_i|^2|x-a_j|^2} \,\varphi^2\, \mu(x)dx
\\&=
\int_{\Uk}\sum_{\substack{i,j=1\\j\ne i}}^n
\frac{\beta^2\,(x-a_i)\cdot (x-a_j)}{|x-a_i|^2|x-a_j|^2} \,\varphi^2\, \mu(x)dx
\\&+
\int_{\Ude\setminus \Uk}\sum_{\substack{i,j=1\\j\ne i}}^n
\frac{\beta^2\,(x-a_i)\cdot (x-a_j)}{|x-a_i|^2|x-a_j|^2} \,\varphi^2\, \mu(x)dx
\\&+
\int_{\R^N\setminus \Ude}\sum_{\substack{i,j=1\\j\ne i}}^n
\frac{\beta^2\,(x-a_i)\cdot (x-a_j)}{|x-a_i|^2|x-a_j|^2} \,\varphi^2\, \mu(x)dx
\\&:=
I_1+I_2+I_3,
\end{split}
\end{equation}
Subsequently we will rewrite the mixed term in the following way.
\begin{equation}\label{mixed term}
\begin{split}
\sum_{\substack{i,j=1\\ i\ne j}}^{n}&\frac{(x-a_i)\cdot (x-a_j)}{|x-a_i|^2|x-a_j|^2}  =
\sum_{\substack{i,j=1\\ i\ne j}}^{n}\frac{|x|^2-x\cdot a_i-x\cdot a_j+a_i\cdot a_j}
{|x-a_i|^2|x-a_j|^2}
\\&=
\sum_{\substack{i,j=1\\ i\ne j}}^{n}\frac{\frac{|x-a_i|^2}{2}+\frac{|x-a_j|^2}{2}-\frac{|a_i-a_j|^2}{2}}
{|x-a_i|^2|x-a_j|^2}
\\&=
\sum_{\substack{i,j=1\\ i\ne j}}^{n}\frac{1}{2}\left( \frac{1}{|x-a_i|^2}+\frac{1}{|x-a_j|^2}
-\frac{|a_i-a_j|^2}{|x-a_i|^2|x-a_j|^2}\right) 
\\&=
(n-1)\sum_{i=1}^{n}\frac{1}{|x-a_i|^2}
-\frac{1}{2}\sum_{\substack{i,j=1\\ i\ne j}}^{n}\frac{|a_i-a_j|^2}{|x-a_i|^2|x-a_j|^2}
\\&=
(n-1)\sum_{i=1}^{n}\frac{1}{|x-a_i|^2}-W.
\end{split}
\end{equation} 
To estimate the integral $I_1$ in (\ref{integr mixed term}) we use the estimate 
(\ref{W from below}) in Section \ref{sec:3} 
for $W$ in a ball centered in $a_k$ and the identity (\ref{mixed term}).  We obtain

\begin{equation*}
\begin{split}
I_1 & \le 
\beta^2\sum_{k=1}^{n} 
\int_{B(a_k,\varepsilon)}\Biggl[
\sum_{i=1}^{n}\frac{n-1}{|x-a_i|^2}-
\frac{n-1}{|x-a_k|^2}\frac{d^2}{(d+\varepsilon)^2}
\\&-
\sum_{\substack{i,j\ne k\\j>i}}\frac{|a_i-a_j|^2}
{(|a_k-a_i|+\varepsilon)^2(|a_k-a_j|+\varepsilon)^2}\Biggr]\,\varphi^2\, \mu(x)dx
\\&=
\beta^2\sum_{k=1}^{n} 
\int_{B(a_k,\varepsilon)}\Biggl\{
\frac{n-1}{|x-a_k|^2}\left[1-\frac{d^2}{(d+\varepsilon)^2}\right]+
\sum_{\substack{i=1\\ i\ne k}}^n \frac{n-1}{|x-a_i|^2}
\\&-
\sum_{\substack{i,j\ne k\\j>i}}\frac{|a_i-a_j|^2}
{(|a_k-a_i|+\varepsilon)^2(|a_k-a_j|+\varepsilon)^2}\Biggr\}\,\varphi^2\, \mu(x)dx.
\end{split}
\end{equation*}
To complete the estimate of $I_1$ we observe that in $B(a_k,\varepsilon)$, for $i\ne k$, it occurs
$$
|x-a_i|\ge |a_k-a_i|-|x-a_k|\ge |a_k-a_i|-\varepsilon
$$
so we get
$$
\sum_{\substack{i=1\\ i\ne k}}^n \frac{n-1}{|x-a_i|^2}\le
\sum_{\substack{i=1\\ i\ne k}}^n \frac{n-1}{(|a_k-a_i|-\varepsilon)^2}.
$$
Then
$$
I_1\le \beta^2\sum_{k=1}^{n} 
\int_{B(a_k,\varepsilon)}\left\{
\frac{n-1}{|x-a_k|^2}\left[1-\frac{d^2}{(d+\varepsilon)^2}\right]+c_3\right\}
\,\varphi^2\,\mu(x)dx,
$$
where
$$
c_3=\sum_{\substack{i=1\\ i\ne k}}^n \frac{n-1}{(|a_k-a_i|-\varepsilon)^2}-
\sum_{\substack{i,j\ne k\\j>i}}\frac{|a_i-a_j|^2}
{(|a_k-a_i|+\varepsilon)^2(|a_k-a_j|+\varepsilon)^2}.
$$
For the second integral $I_2$ we observe that in $B(a_k,\varepsilon+\delta)\setminus
B(a_k,\varepsilon)$, for $j\ne k$, $|x-a_k|>\varepsilon$ and
$$
|x-a_j|\ge |a_k-a_j|-|x-a_k|\ge |a_k-a_j|-(\varepsilon+\delta)
$$
Therefore
\begin{equation*}
\begin{split}
I_2\le &\int_{\Ude\setminus \Uk}\sum_{\substack{i,j=1\\j\ne i}}^n
\frac{\beta^2}{|x-a_i||x-a_j|} \,\varphi^2\, \mu(x)dx
\\&\le
\frac{ n\beta^2}{\varepsilon}\sum_{\substack{j=1\\j\ne k}}^n
\frac{1}{|a_k-a_j|-(\varepsilon+\delta)}\int_{\Ude\setminus \Uk}
\varphi^2\, \mu(x)dx.
\end{split}
\end{equation*}
The remaining integral $I_3$ can be can be estimated as follows.
\begin{equation*}
\begin{split}
I_3\le &\int_{\R^N\setminus\Ude}\sum_{\substack{i,j=1\\j\ne i}}^n
\frac{\beta^2}{|x-a_i||x-a_j|} \,\varphi^2\, \mu(x)dx
\\&\le
\frac{ n(n-1)\beta^2}{(\varepsilon+\delta)^2}\int_{\R^N\setminus\Ude}
\varphi^2\, \mu(x)dx.
\end{split}
\end{equation*}
Starting from (\ref{integr mixed term}) and using the estimates 
obtained for $I_1$, $I_2$ and $I_3$, we get for $\varepsilon$ small enough,
\begin{equation}\label{ineq mixed term}
\begin{split}
\int_{\R^N}& \sum_{\substack{i,j=1\\j\ne i}}^n
\frac{\beta^2\,(x-a_i)\cdot (x-a_j)}{|x-a_i|^2|x-a_j|^2} \,\varphi^2\, \mu(x)dx
\\&\le
\int_{\R^N}
\sum_{i=1}^{n} \frac{\beta^2(n-1)c_\varepsilon}{|x-a_i|^2}\,\varphi^2\,\mu(x)dx
+c_4 \int_{\R^N}\,\varphi^2\,\mu(x)dx,
\end{split}
\end{equation}
where
$$
c_\varepsilon=1-\frac{d^2}{(d+\varepsilon)^2}
\qquad \hbox{and}\qquad
c_4=\frac{ n\beta^2}{\varepsilon}\sum_{\substack{j=1\\j\ne k}}^n
\frac{1}{|a_k-a_j|-(\varepsilon+\delta)}.
$$
Going back to (\ref{first part}), by (\ref{ineq mixed term})
and by the hypothesis $H_2)$, we deduce that

\begin{equation*}
\begin{split}
\int_{\R^N}\sum_{i=1}^{n}&
\frac{\beta(N+K_\mu-2)  -\beta^2\left[1+(n-1)c_\varepsilon\right] }{|x-a_i|^2}\varphi^2 \, \mu(x)dx
\\&\le
\int_{\R^N}^{}|\nabla \varphi|^2 \mu(x)dx 
+\left( c_4 +C_\mu\right)
\int_{\R^N} \varphi^2\, \mu(x)dx.
\end{split}
\end{equation*}

The maximum of the function $\beta\mapsto c=(N+K_\mu-2)\beta-\beta^2
\left[1+(n-1)c_\varepsilon\right]$, fixed $\varepsilon$, 
is $c_{max}(N+K_\mu)=\frac{(N+K_\mu-2)^2}{4[1+(n-1)c_\varepsilon]}$
attained in $\beta_{max}=\frac{N+K_\mu-2}{2[1+(n-1)c_\varepsilon]}$.

\end{proof}

We conclude with some remarks.
If $\varepsilon$ tends to zero, and then if we get close enough to the single pole,
the constant $c=\frac{(N+K_\mu-2)^2}{4[1+(n-1)c_\varepsilon]}$ tends to the optimal constant 
$c_o(N+K_\mu)$. The constant $k= c_4+C_\mu$ , $c_4$ with $\beta=\beta_{max}$, 
is better than the analogous constant in \cite{CPT2}.

In the case of Gaussian measure the constant $K_\mu$  tends to zero as the radius $\varepsilon$
of the sphere centered in a single pole tends to zero (cf. example in Section \ref{sec:2}). 

Finally, we observed that 
as a consequence of Theorem \ref{wHi}, we deduce the estimate 
$$
\| V^{\frac{1}{2}}\varphi\|_{L_\mu^2(\R^N)}\le c \|\varphi\|_{H^1_\mu(\R^N)},
$$
with $V=\sum_{i=1}^n \frac{1}{|x-a_i|^2}$  and 
$c$  a constant independent of $V$ and $\varphi$.

For $L^p$ estimates and embedding results of this type 
with some applications to elliptic equations see, for example, 
\cite{CAJIM2008,CAMIA,CAembedding}.


\end{document}

\